\documentclass[12pt,a4paper]{amsart}
\usepackage[totalwidth=15.75cm,totalheight=22.275cm]{geometry}
\usepackage{amsmath,amsfonts,amsthm,amssymb,verbatim,enumerate,quotes,graphicx,mathtools}
\usepackage{color}

\newcommand{\comments}[1]{}
\newtheorem{theorem}{Theorem}
\newtheorem{lemma}{Lemma}

\newtheorem*{remark}{Remark}

\newtheorem{conjecture}{Conjecture}

\def \isnatural {\in\mathbb{N}}
\def\R{\mathbb{R}}
\def\C{\mathbb{C}}
\def\N{\mathbb{N}}

\def\H{\mathbb{H}}
\newcommand{\tef}{transcendental entire function}
\newcommand\qfor{\quad\text{for }}
\makeatletter
\def\blfootnote{\xdef\@thefnmark{}\@footnotetext}
\makeatother
\begin{document}
%
%
%
%
\title[The bungee set in quasiregular dynamics]{The bungee set in quasiregular dynamics}
\author{Daniel A. Nicks, \, David J. Sixsmith}
\address{School of Mathematical Sciences \\ University of Nottingham \\ Nottingham
NG7 2RD \\ UK \\ ORCiD:0000-0002-9493-2970}
\email{Dan.Nicks@nottingham.ac.uk}
\address{Dept. of Mathematical Sciences \\
	 University of Liverpool \\
   Liverpool L69 7ZL\\
   UK \\ ORCiD: 0000-0002-3543-6969}
\email{djs@liverpool.ac.uk}
%
%
%
%
\begin{abstract}
In complex dynamics, the bungee set is defined as the set points whose orbit is neither bounded nor tends to infinity. In this paper we study, for the first time, the bungee set of a quasiregular map of transcendental type. We show that this set is infinite, and shares many properties with the bungee set of a {\tef}. By way of contrast, we give examples of novel properties of this set in the quasiregular setting. In particular, we give an example of a quasiconformal map of the plane with a non-empty bungee set; this behaviour is impossible for an analytic homeomorphism.
\end{abstract}
\maketitle
%
%
%
%
\section{Introduction}
Suppose that $f$ is an entire function. In the study of complex dynamics it is common to partition the complex plane into two sets. Firstly, the \emph{Julia set} $J(f)$, which consists of points in a neighbourhood of which the iterates of $f$ are, in some sense, chaotic. Secondly, its complement the \emph{Fatou set} $F(f) := \C \setminus J(f)$. For more information on complex dynamics, including precise definitions of these sets, we refer to \cite{MR1216719}.

An alternative partition divides the plane into three sets based on the nature of the orbits of points; the \emph{orbit} of a point $z$ is the sequence $(f^n(z))_{n\geq 0}$ of its images under the iterates of $f$. This partition is defined as follows:
\begin{itemize}
\item The \emph{escaping set} $I(f)$ consists of those points whose orbit tends to infinity.
\item The \emph{bounded orbit set} $BO(f)$ consists of those points whose orbit is bounded.
\item The \emph{bungee set} $BU(f) := \C \setminus (I(f) \cup BO(f))$ contains all other points. 
\end{itemize}

Suppose that $P$ is a polynomial of degree greater than one. Then the escaping set $ I(P) $ is the basin of attraction of infinity, and so $I(P) \subset F(P)$. The set $BO(P)$ (usually in this context denoted by $K(P)$) is known as the \emph{filled Julia set} and has been extensively investigated, since $J(P) = \partial BO(P)$. It is well-known that $BU(P)$ is empty in this case.

The escaping set for a general \tef\ $ f $ was first studied by Eremenko \cite{MR1102727}, and has been the focus of much subsequent research in complex dynamics. 
The set $BO(f)$ for a {\tef} $f$ was studied in \cite{MR2869069} and \cite{MR3118409}. 
If $f$ is transcendental, then $BU(f)$ is non-empty; indeed the Hausdorff dimension of $BU(f) \cap J(f)$ is greater than zero \cite[Theorem 5.1]{OsSix}. The properties of $BU(f)$ were studied in \cite{OsSix} and subsequently in \cite{Lazebnik2017611, DaveDynamical}. Examples of {\tef}s with Fatou components in $BU(f)$ were given in \cite{Bish3, MR918638, Lazebnik2017611, FJL}. These sets are connected by the equation \cite{MR1102727, MR3118409, OsSix}
\begin{equation}
\label{boundaries}
J(f) = \partial BU(f) = \partial I(f) = \partial BO(f).
\end{equation}

To move the study of the bungee set into a more general setting, we consider the iteration of quasiregular and quasiconformal maps;
%
we refer to \cite{MR1238941, MR950174} for definitions. 
%
Suppose that $d \geq 2$, and that $f : \R^d \to \R^d$ is a quasiregular map with an essential singularity at infinity, in which case we say that $f$ is of \emph{transcendental type}. Following \cite{MR3009101, MR3265283}, we define the \emph{Julia set $J(f)$} as the set of all $x \in \R^d$ such that
\begin{equation*}
\operatorname{cap} \left(\R^d\backslash \bigcup_{k=1}^\infty f^k(U)\right) = 0,
\end{equation*}
for every neighbourhood $U$ of $x$. Here if $S \subset \R^d$, then we write cap $S = 0$ if $S$ has \emph{zero (conformal) capacity}, and otherwise we write cap $S > 0$. Again, we refer to \cite{MR1238941, MR950174} for a definition and properties of conformal capacity.

%
%

It is known that if $f$ is a quasiregular map of transcendental type, then the Julia set is infinite \cite[Theorem 1.1]{MR3265283}. It is easy to see that $J(f)$ is closed, and also that $J(f)$ is \emph{completely invariant}, in the sense that $x \in J(f)$ if and only if $f(x) \in J(f)$.

The definitions of $I(f), BO(f)$ and $BU(f)$ can be modified in an obvious way to apply to quasiregular maps of $\R^d$. In the quasiregular setting, the escaping set has been studied in \cite{MR2448586,MR3265357,MR3215194,danslow}, and the bounded orbit set in \cite{MR3265283}.  Our goal in this paper is to study $BU(f)$ in the case that $f$ is quasiregular and of transcendental type. 
%
Our first result shows that the bungee set of a quasiregular map of transcendental type is never empty, and in fact always meets the Julia set.
\begin{theorem}
\label{theo:itsinfinite}
Suppose that $f : \R^d \to \R^d$ is a quasiregular map of transcendental type. Then $BU(f) \cap J(f)$ is an infinite set.
\end{theorem}

We now specialise to the case that the Julia set has positive capacity. In fact there are no known examples where the Julia set of a quasiregular map of transcendental type does not have positive capacity, and the following conjecture arises from \cite{MR3009101, MR3265283}.
\begin{conjecture}
\label{con1}
Suppose that $f : \R^d \to \R^d$ is a quasiregular map of transcendental type. Then cap $J(f) > 0$.
\end{conjecture}
The next three theorems are the main results of this paper. The first two show that, for quasiregular maps of transcendental type, the first equality of \eqref{boundaries} need not hold in general, but we are guaranteed inclusion provided that the Julia set has positive capacity.
\begin{theorem}
\label{theo:capnonzero}
Suppose that $f : \R^d \to \R^d$ is a quasiregular map of transcendental type. If cap $J(f) >0$, then $BU(f) \cap J(f)$ is an infinite set and
\begin{equation}
\label{e:inclusions}
J(f) \subset \partial BU(f) \cap \partial I(f) \cap \partial BO(f).
\end{equation}
\end{theorem}
\begin{theorem}
\label{theo:JnotboundaryBU}
There is a quasiregular map of transcendental type $f : \R^2 \to \R^2$ such that cap $J(f) > 0$ and $J(f) \ne \partial BU(f)$.
\end{theorem}
Our proof of Theorem~\ref{theo:JnotboundaryBU} relies on the following, perhaps somewhat surprising, result.
\begin{theorem}
\label{theo:quasiconformalexample}
There is a quasiconformal map $f : \R^2 \to \R^2$ such that $BU(f)\ne\emptyset$.
\end{theorem}
If $f : \R^2 \to \R^2$ is an analytic homeomorphism, in other words an affine map, the dynamics of $f$ are not particularly interesting; certainly we have that $BU(f) = \emptyset$. Theorem~\ref{theo:quasiconformalexample} shows that this is not the case for quasiconformal maps of the plane.
\begin{remark}\normalfont
Suppose that $f : \R^d \to \R^d$ is a quasiregular map \emph{not} of transcendental type. Suppose also that the degree of $f$ is sufficiently large compared to the distortion of $f$; in technical terms we require that deg $f > K_I(f)$. It is shown in \cite[p.28]{MR2755919} (see also \cite{ETS:9408364}) that $I(f)$ contains a neighbourhood of infinity, and so $BU(f)$ is empty.
\end{remark}
Finally, returning to Conjecture~\ref{con1},  we note that there are many conditions known to be sufficient for Conjecture~\ref{con1} to hold. For example, the Julia set of a quasiregular map of transcendental type $f : \R^2 \to \R^2$ is always of positive capacity \cite[Theorem 1.11]{MR3265283}, so this part of Theorem~\ref{theo:JnotboundaryBU} is immediate. The paper \cite{MR3265283} gives many other sufficient conditions; for example, if $f$ is locally Lipschitz or has bounded local index. In the following we add to this list a simple condition on the growth of the function; roughly speaking, all functions that do not grow too slowly have a Julia set of positive capacity. Here, for $r > 0$, we define the \emph{maximum modulus function} by
\[
M(r, f) :=  \max_{|x| = r} |f(x)|.
\]
\begin{theorem}
\label{theo:capJ}
Suppose that $f : \R^d \to \R^d$ is a quasiregular map of transcendental type. Suppose also that
\begin{equation}
\label{eq:grows}
\liminf_{r\rightarrow\infty} \frac{\log \log M(r, f)}{\log \log r} = \infty.
\end{equation}
Then cap $J(f)>0$.
\end{theorem}
\begin{remark}\normalfont
A quasiregular map $f : \R^d \to \R^d$ has \emph{positive lower order} if there exist $r_0 > 0$ and $\epsilon > 0$ such that
\[
M(r, f) > \exp r^\epsilon, \qfor r \geq r_0.
\]
It is easy to see that a quasiregular map with positive lower order satisfies \eqref{eq:grows}.
\end{remark}
\subsection*{Notation}
For $0 < r_1 < r_2$, we denote the spherical shell centred at the origin by
\[
A(r_1,r_2) := \{x \in \R^d : r_1 < |x| < r_2\},
\]
and the ball with centre at the origin and radius $r_1$ by
\[
B(r_1) := \{x \in \R^d : |x| < r_1\}.
\]

Finally, if $S \subset \R^d$, then we denote the boundary of $S$ in $\R^d$ by $\partial S$, and closure of $S$ in $\R^d$ by $\overline{S}$.

%
%
%
\section{Proof of Theorem~\ref{theo:itsinfinite} and Theorem~\ref{theo:capnonzero}}
\label{S:itsinfinite}
We use the following result. This is a version of \cite[Lemma 3.1]{Sixsmithmax} stated for quasiregular maps. The proof is omitted, as it is almost identical to the proof of the original. 
\begin{lemma}
\label{lemm:l1}
Suppose that $(E_n)_{n\isnatural}$ is a sequence of compact sets in $\R^d$ and $(m_n)_{n\isnatural}$ is a sequence of integers. Suppose also that $f : \R^d \to \R^d$ is a quasiregular map such that $E_{n+1} \subset f^{m_n}(E_n )$, for $n\isnatural$. Set $p_n := \sum_{k=1}^n m_k$, for $n\isnatural$. Then there exists $\zeta\in E_1$ such that 
\begin{equation}
\label{feq}
f^{p_n}(\zeta) \in E_{n+1}, \qfor n\isnatural.
\end{equation}
If, in addition, $E_n \cap J(f) \ne \emptyset$, for $n\isnatural$, then there exists $\zeta \in E_1 \cap J(f)$ such that (\ref{feq}) holds.
\end{lemma}

We need the following, which is taken from \cite[Lemma 3.3]{danslow} and \cite[Lemma~3.4]{danslow}. Here a quasiregular map $f : \R^d \to \R^d$ of transcendental type has the \emph{pits effect} if there exists $n \in \N$ such that, for all $c > 1$ and $\epsilon > 0$, there exists $r_0$ such that if $r > r_0$, then the set
\[
\{x \in \R^d : r \leq |x| \leq cr, |f(x)| \leq 1\}
\]
can be covered by $n$ balls of radius $\epsilon r$.
\begin{lemma}
\label{lem:dan}
Suppose that $f : \R^d \to \R^d$ is a quasiregular map of transcendental type that has the pits effect. Then there exist increasing sequences of positive real numbers $(s_n)_{n\in\N}$ and $(t_n)_{n\in\N}$, both tending to infinity, such that, for $t \geq t_n$, 
\begin{equation}
\label{eq:dan}
f(A(s_n , t)) \supset B(2t), \qfor n \in \N.
\end{equation}
\end{lemma}
Note that \cite[Lemma 3.4]{danslow} states $f(A(s_n , t)) \supset A(s_n ,2t)$ in place of \eqref{eq:dan}. Our stronger statement is easily derived from the proof of \cite[Lemma 3.4]{danslow}. 

\begin{proof}[Proof of Theorem~\ref{theo:itsinfinite} and Theorem~\ref{theo:capnonzero}]
Suppose that $f : \R^d \to \R^d$ is a quasiregular map of transcendental type. The proof splits into two cases: cap $J(f) > 0$ and cap $J(f) = 0$.

We consider first the case that cap $J(f) > 0$. Pick $R>0$ sufficiently large that cap~$J' > 0$, where $J' := J(f) \cap B(R)$. For each $n \in \N$ set 
\[
J_n := J(f) \cap \{ x \in \R^d : |x| > n \}.
\]
It follows from \cite[Theorem~1.2]{MR583633}, which is the quasiregular analogue of Picard's great theorem, together with complete invariance, that $J(f) \setminus f(J_n)$ is a finite set, for $n \in \N$, and so has capacity zero. If cap $J_n = 0$, then cap $f(J_n) = 0$ (see, for example, \cite[Theorem 10.15]{MR950174}) and so cap $J(f) \setminus f(J_n) > 0$. This is a contradiction. Hence cap $J_n > 0$, for $n \in \N$. 

Choose a point $x_1 \in J(f)$, and let $U_1$ be a neighbourhood of $x_1$ of diameter at most one. It follows from the definition of the Julia set that cap $(\R^d \setminus \bigcup_{k\in\N} f^k(U_1)) = 0$, and so there exist $m_1 \in \N$ and $x_1' \in U_1$ such that 
\[ 
x_2 := f^{m_1}(x_1') \in J_2.
\]
Let $U_1' \subset U_1$ be a neighbourhood of $x_1'$ sufficiently small that $U_2 := f^{m_1}(U_1')$ is of diameter at most one.

Now, since cap $J' > 0$, and $U_2$ is open and meets $J(f)$, there exist  $m_2 \in \N$ and $x_2' \in U_2$ such that 
\[ 
x_3 := f^{m_2}(x_2') \in J'.
\]
Let $U_2' \subset U_2$ be a neighbourhood of $x_2'$ sufficiently small that $U_3 := f^{m_2}(U_2')$ is of diameter at most one.

Continuing inductively, we obtain a sequence of domains $(U_n)_{n\in\N}$, each of diameter at most one, and a sequence of integers $(m_n)_{n\in\N}$ such that $f^{m_n}(U_n) \supset U_{n+1}$, and $U_n$ meets $J_n$ when $n$ is even, and $J'$ when $n\geq 3$ is odd.

An application of Lemma~\ref{lemm:l1} gives that there is a point 
\[
\xi \in \overline{U_1} \cap BU(f) \cap J(f).
\] 
Since $f^n(\xi) \in BU(f) \cap J(f)$, for $n \geq 0$, we obtain that $BU(f) \cap J(f)$ is infinite.

Since $x_1$ and $U_1$ were arbitrary, it follows that $J(f) \subset \overline{BU(f)}$. It is known that $J(f) \subset \partial I(f) \cap \partial BO(f)$ \cite[Theorem 1.3]{MR3265283}. Thus $J(f) \subset \partial BU(f)$, and so \eqref{e:inclusions} holds. This completes the proof of Theorem~\ref{theo:capnonzero}, and also of Theorem~\ref{theo:itsinfinite} when cap $J(f) > 0$. \\

It remains to prove Theorem~\ref{theo:itsinfinite} in the case that cap $J(f) = 0$, so we now assume that the Julia set has capacity zero. It follows by \cite[Corollary 1.1]{MR3265283} that $f$ has the pits effect.

Let $(s_n)_{n\in\N}$ and $(t_n)_{n\in\N}$ be as given in Lemma~\ref{lem:dan}. Set $V_n := A(s_n, t_n)$, for $n \in \N$. We may assume that $B(2t_n)$ meets $J(f)$ for all $n \in \N$, so \eqref{eq:dan} and complete invariance imply that
\[
V_n \cap J(f) \ne \emptyset, \qfor n\in \N.
\]

By \eqref{eq:dan} again,
\[
f(V_n) \supset B(2t_n) \supset V_1, \qfor n \in \N,
\]
and, moreover, if $m_n \in \N$ is sufficiently large that $2^{m_n} \geq t_n/t_1$, then
\[
f^{m_n}(V_1) \supset B(2^{m_n}t_1) \supset V_n, \qfor n \in \N.
\]

An application of Lemma~\ref{lemm:l1} (with $E_n = \overline{V_1}$ for odd $n$, and $E_n = \overline{V_n}$ for even $n$) gives that there is a point 
\[
\xi \in \overline{V_1} \cap BU(f) \cap J(f),
\]
because we have forced oscillation of the orbit. As earlier, it follows that $BU(f) \cap J(f)$ is infinite.
\end{proof}
%
%
%
%
\section{Examples}
\label{S:quasiconformalexamples}
In this section we first prove Theorem~\ref{theo:quasiconformalexample}, and then use the function constructed to prove Theorem~\ref{theo:JnotboundaryBU}.

\begin{proof}[Proof of Theorem~\ref{theo:quasiconformalexample}]
From now on we identify $\R^2$ with $\C$ in the obvious way. We construct a quasiconformal map $f : \C \to \C$ such that $BU(f)~\ne~\emptyset$. First we fix $y_0 > 100$, and let $T_0$ be the domain
$$
T_0 := \{ x + iy : y > y_0, \ |x| < 1/y \}.
$$
We define a continuous map $\psi : \overline{T_0} \to \overline{T_0}$ as follows. If $x + iy \in \overline{T_0}$, then we set
\begin{equation}
\label{psidef}
\psi(x + iy) := \frac{xy}{y + 1/y - |x|} + i(y + 1/y - |x|).
\end{equation}
Note that $\psi$ is the identity map on the two vertical sides of $T_0$. Note in addition that  
\begin{equation}
\label{toinf}
\psi^n(z) \rightarrow\infty \text{ as } n\rightarrow\infty, \qfor z = 0 + iy \text{ where } y > y_0.
\end{equation}

We show that $\psi$ is quasiconformal on $T_0$ by estimating the derivative. By differentiating \eqref{psidef} we obtain that, as $y \rightarrow\infty$, 
\begin{equation*}
D\psi(x+iy) = 
\left(
\begin{array}{cc}%
	1 + O(y^{-2}) & O(y^{-2}) \\
	\pm 1 & 1 + O(y^{-2})
\end{array}
\right), \qfor (x + iy) \in T_0.
\end{equation*} 
It follows that $\psi$ is indeed quasiconformal on $T_0$.

Roughly speaking $\overline{T_0}$ is an infinite ``straight snake''. We now seek to define a quasiconformal map $\phi$ on $T_0$, homeomorphic up to the boundary, such that $\overline{\phi(T_0)}$ is a ``coiled snake''. Moreover half the bends in this snake will have imaginary parts tending to infinity, whereas the remaining bends will be within a fixed distance of the origin. 

To construct this map, we first need to fix two particular quasiconformal maps. Let $A$ be the rectangle 
\[ 
A := \{ z : \operatorname{Re}(z) \in [0, 1], \operatorname{Im}(z) \in [0, 2]\},
\]
and let $B$ be the half-annulus 
\[ 
B := \{ z : \operatorname {Im}(z) \geq 0, 1/2 \leq |z - 3/2| \leq 3/2 \}.
\] 
We define a map $\nu_r : A \to B$ by
\begin{equation}
\label{nurdef}
\nu_r(x+iy) := 3/2 + (x-3/2)e^{-i \pi y/2}.
\end{equation}

It can be checked that $\nu_r$ is a quasiconformal map on the interior of $A$. It is also easy to check that $\nu_r$ is the identity on the lower boundary of $A$, maps each vertical line segment ending at a point on the lower boundary of $A$ to a semi-circle in $B$, and maps the upper boundary of $A$ to the right-hand lower boundary of $B$ by an affine transformation. 

The second quasiconformal map is 
\begin{equation}
\label{nuldef}
\nu_l(x+iy) := -1/2 + (x+1/2)e^{i \pi y/2}.
\end{equation}
This maps $A$ to the half annulus \[ \{ z : \operatorname{Im}(z) \geq 0, 1/2 \leq |z + 1/2| \leq 3/2 \},\] once again fixing the lower boundary of $A$.

Let the sequences $(s_n)_{n\in\N}$ and $(t_n)_{n\in\N}$ of positive real numbers be defined by $t_n := 2^n$, $s_1 := y_0$ and then 
\[ 
s_{n+1} := s_n + 2t_n + 4/(s_n + t_n) + 4/(s_n + 2t_n + 4/(s_n + t_n)).
\] 
Note that
\begin{equation}
\label{tsum}
\sum_{n=0}^\infty 1/t_n < \infty.
\end{equation}
Roughly speaking $t_n$ will be the height of the $n$th bend of the snake, and $s_n$ will measure the total distance along the snake to the start of the $n$th bend. Note that $s_{n+1}$ is only approximately equal to $s_n + 2t_n$; the additional terms correspond to the ``corners'' of the bends. See Figure~\ref{f1}.

We now divide the set $\overline{T_0}$ into infinitely many collections of four closed approximate rectangles. In particular, for each $n \in \N$ we define:
\begin{itemize}
\item A strip of height $t_n$ given by $$S^1_n := \overline{T_0} \cap \{ x + iy : s_n \leq y \leq s_n + t_n \}.$$

\item A small (approximate) rectangle, of height twice its width, given by $$S^2_n :=\overline{T_0} \cap \{ x + iy : s_n + t_n \leq y \leq s_n + t_n + 4/(s_n + t_n) \}.$$ 

\item A second strip of height $t_n$ given by $$S_n^3 := \overline{T_0} \cap \{ x + iy : s_n + t_n + 4/(s_n + t_n) \leq y \leq s_n + 2t_n + 4/(s_n + t_n)\}.$$

\item A second (approximate) rectangle, also of height twice its width, given by \[S_n^4 := \overline{T_0} \cap \{ x + iy : s_n + 2t_n + 4/(s_n + t_n) \leq y \leq s_{n+1} \}.\] 
\end{itemize}

We define $\phi$ by specifying it first on $S^1_1$, then on $S^2_1$, then on $S^3_1$, and so on ``up'' $T_0$. Note that the rectangles above meet where upper and lower boundaries coincide, but we will ensure that the definitions of $\phi$ respect this. In addition, the upper and lower boundaries will be mapped only by affine transformations. 

First we define $\phi$ on the lowest collection of four rectangles in $T_0$.
\begin{itemize}
\item On $S^1_1$ we let $\phi$ be the identity.

\item The action of $\phi$ on $S^2_1$ is defined as follows. First translate $S^2_1$ so that its bottom left corner lies at the origin. Then enlarge it by a scale factor of $(s_1+t_1)/2$, so that it maps into $A$, and then map it by the function $\nu_r$ defined in \eqref{nurdef}. Then scale it by a scale factor of $2/(s_1+t_1)$, and translate it so the left-hand lower boundary of the image coincides with the upper boundary of $\phi(S^1_1)$. (Observe here that the enlarged translation of $S^2_1$ is only a subset of the rectangle $A$. This does not affect the argument).
 
\item The action of $\phi$ on $S^3_1$ is defined by first rotating by one half-turn, and then translating so that the upper boundary of the image of $S_1^3$ coincides with the right-hand lower boundary of $\phi(S_1^2)$.

\item The action of $\phi$ on $S^4_1$ is defined as follows, and is very similar to the action on $S^2_1$. First translate $S^4_1$ so that its bottom left corner lies at the origin. Then enlarge it by a scale factor of $(s_n + 2t_n + 4/(s_n + t_n))/2$ to obtain a subset of $A$. Then apply the map $\nu_l$ defined in \eqref{nuldef}, followed by an second scaling with scale factor equal to $2/(s_n + 2t_n + 4/(s_n + t_n))$. Finally rotate by one half-turn, and then translate so that the upper left-hand boundary of the image of $S_1^4$ coincides with the lower boundary of $\phi(S_1^3)$.

\end{itemize}

\begin{figure}
	\includegraphics[width=16cm,height=10cm]{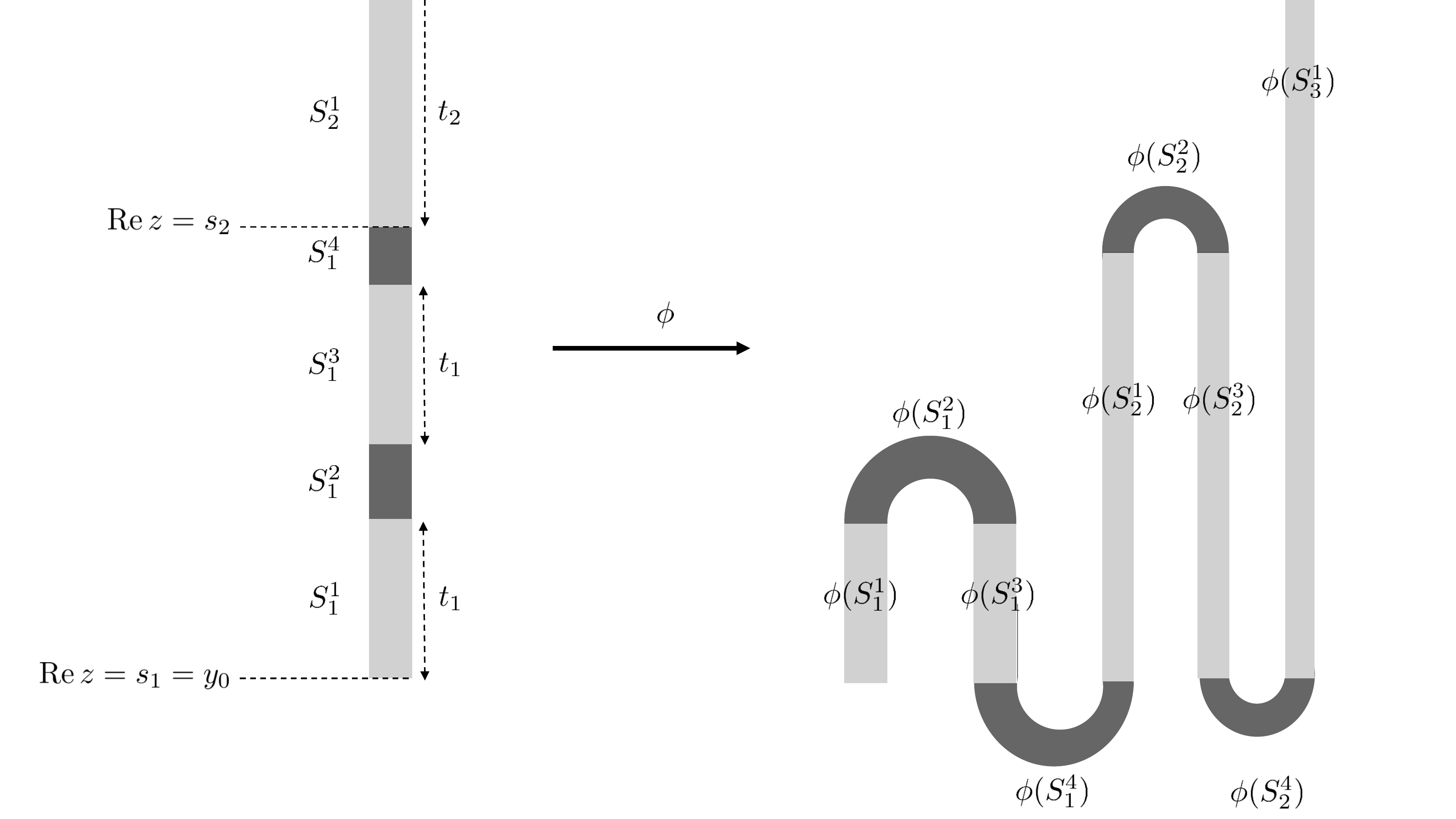}
  \caption{A rough schematic of the construction of the map $\phi$.
	}\label{f1}
\end{figure}

It is now clear how to continue this process; we iterate the four steps above, although with different translations at each stage to ensure continuity at the boundary. In particular, for each $n \geq 2$, $\phi$ maps $S^1_n$ by a translation, rather than the identity. See Figure~\ref{f1}. Note that it follows from \eqref{tsum} that the snake remains within a strip of bounded real part.

In order to see that $\phi$ is quasiconformal on $T_0$, we now check that subsequent sections of the snake do not overlap; that is, for each $n \in \N$, the sets $\phi(S_n^1)$, $\phi(S_n^3)$ and $\phi(S_{n+1}^1)$ are pairwise disjoint. To see this, fix $n \in \N$. Note that the base of the strip $\phi(S_n^1)$ is of width $2/s_n$, and the top of this strip is of width $2/(s_n + t_n)$. Also, by construction, the left-hand side of the strip $\phi(S_n^3)$ is at least $4/(s_n+t_n)$ from the left-hand side of the strip $\phi(S_n^1)$. Now, it follows from the definitions that $t_n < s_n$, and hence $2/s_n < 4/(s_n + t_n)$. Thus the strips $\phi(S_n^1)$ and $\phi(S_n^3)$ are pairwise disjoint. The proof that the strips $\phi(S_n^3)$ and $\phi(S_{n+1}^1)$ are also pairwise disjoint is similar and is omitted.

We are now able to define our quasiconformal map $f : \C \to \C$. First, set $\widetilde{T} := \phi(T_0)$. For $z \in \widetilde{T}$ we define $f(z) := (\phi \circ \psi \circ \phi^{-1})(z)$. It is easy to check that $f$ is quasiconformal on $\widetilde{T}$ and extends to the identity on all parts of the boundary of $\widetilde{T}$ apart from the line segment $\{ x + iy : y = y_0, \ |x| < 1/y_0 \}$.
 
We then extend $f$ to a map of the whole plane. First we let $R$ be the rectangle 
\[ R := \{ x + iy : y \in (0, y_0), \ |x| < 1/y_0 \}. \]

On $\C \setminus (\widetilde{T} \cup R)$ we let $f$ be the identity map. It is then straightforward, using, for example, \cite[Theorem 6]{SixsmithNicks2}, to see that $f$ can be extended to a quasiconformal map of the whole plane. Note that we are actually only interested in the behaviour of $f$ in $\widetilde{T}$; the rectangle $R$ is only used to allow us to extend the definition of $f$ to the whole plane.

It is now straightforward to see, by \eqref{toinf} and the geometry of $\widetilde{T}$, that 
\[ \phi(\{ x + iy : x = 0, y > y_0 \}) \subset BU(f), \]
and this completes the construction.
\end{proof}
Finally we prove Theorem~\ref{theo:JnotboundaryBU} by constructing a quasiregular map $h : \C \to \C$, of transcendental type, such that $\partial BU(h) \setminus J(h) \ne \emptyset$.
\begin{proof}[Proof of Theorem~\ref{theo:JnotboundaryBU}]
We first use a technique from \cite[Section 6]{MR2448586}, (see also \cite[Section 4]{MR3008885}), to define a quasiregular map $g : \C \to \C$ of transcendental type which is equal to the identity in the upper half-plane $\H$. 

In particular we choose $\delta > 0$ small, and then set
\[
g(z) := 
\begin{cases}
z, &\text{for } \operatorname{Im} z \geq 0, \\
z - \delta (\operatorname{Im} z) \exp(-z^2), &\text{for } \operatorname{Im} z \in [-1, 0), \\
z + \delta \exp(-z^2), &\text{otherwise}.
\end{cases}
\]
It can be shown by a calculation that if $\delta$ is sufficiently small, then $g$ is quasiregular. It is clearly of transcendental type. 

Now, let $f$ be the quasiconformal map constructed in the proof of Theorem~\ref{theo:quasiconformalexample}. We note that the ``snake'' $\widetilde{T}$ constructed in the proof of that result lies in $\H$. We set $h := g \circ f$. 

Since $f(\H) \subset \H$, we have that $h(\H) \subset \H$, and so $\H \cap J(h) = \emptyset$. Since $g$ is the identity on $\widetilde{T}$, the maps $f$ and $h$ have the same dynamics on $\widetilde{T}$. It follows that $$\H \cap BO(h) \ne \emptyset \quad\text{and}\quad \H \cap BU(h) \ne \emptyset.$$ Hence, in particular, $\H$ meets $\partial BU(h) \setminus J(h)$.
\end{proof}
%
%
%
\section{Proof of Theorem~\ref{theo:capJ}}
\label{S:capJ}
Suppose that $f : \R^d \to \R^d$ is a quasiregular map of transcendental type. It is known that if \eqref{eq:grows} holds, then $J(f) = \partial A(f)$ \cite[Theorem 1.2]{MR3265357}. Here $A(f)$ is the \emph{fast escaping set}, which is a subset of the escaping set consisting of points that iterate to infinity at a rate comparable to iteration of the maximum modulus; the exact definition is not needed here. 

Now, the set $A(f)$ contains continua \cite[Theorem 1.2]{MR3215194}, and so has positive capacity. Moreover, the complement of $A(f)$ contains $BO(f)$, and so also has positive capacity \cite[Theorem 1.4]{MR3265283}. 

Suppose that cap $\partial A(f) = 0$. It follows by \cite[Corollary 2.2.5]{MR1238941} that $\partial A(f)$ is totally disconnected, and so $\R^d \setminus \partial A(f)$ is connected. Hence either $A(f) \subset \partial A(f)$ or $\R^d \setminus A(f) \subset \partial A(f)$. This is impossible, as a set of positive capacity cannot be contained in a set of zero capacity. Hence cap $J(f) = $ cap $\partial A(f) > 0,$ as required. \\

%
%
%
%
%
%
%

\emph{Acknowledgment:} The authors are grateful to the referee for many helpful comments.
%
%
%
%
%
%

\end{document}